\documentclass{amsart}
\usepackage{amssymb,latexsym}

\theoremstyle{plain}
\newtheorem{theorem}{Theorem}[section]
\newtheorem{corollary}{Corollary}[section]

\newtheorem{lemma}{Lemma}[section]
\newtheorem{proposition}{Proposition}

\theoremstyle{definition}

\newtheorem{example}{Example}
\theoremstyle{remark}
\newtheorem{remark}{Remark}

\numberwithin{equation}{section}

\begin{document}
\title[Semi-harmonicity and Integral Means]
 {Semi-harmonicity, Integral Means \\
and Euler Type Vector Fields}
\author{Chia-chi Tung}
\address{Dept. of Mathematics and
Statistics,\\
         Minnesota State University, Mankato\\
        Mankato, MN 56001} 
\email{chia.tung@mnsu.edu}
\thanks{Supports by Minnesota State University, Mankato 
and the Grant "Globale Methoden in der komplexen Geometrie" 
of the German research society DFG are gratefully 
acknowledged.} 
\date{July 3, 2015}

\begin{abstract}
The Dirichlet product of functions on a semi-Riemann 
domain and generalized Euler vector fields, which 
include the radial, $\bar \partial$-Euler, and the 
$\bar \partial$-Neumann vector fields, are 
introduced. The integral means and the harmonic 
residues of functions on a Riemann domain are 
studied. The notion of semi-harmonicity of 
functions on a complex space is introduced. 
It is shown that, on a Riemann domain, the 
semi-harmonicity of a locally forwardly 
$L^2$-function is characterized by local 
mean-value properties as well as by 
weak-harmonicity. In particular, the Weyl's 
Lemma is extended to a Riemann domain. 
\end{abstract}

\maketitle

\section{Introduction}\label{S:intro.}

\footnotemark\footnotetext{
This paper is a corrected version of 
"Semi-Harmonicity, Integral Means
and Euler Type Vector Fields, Adv. 
Appl. Clifford Alg. 17 (2007), 
555-573". This publication is 
available at Springer via 
http://dx.doi.org/DOI 10.1007/s00006-007-0036-9. 
The author would like to apologize 
to its reader for this inconvenience.}%
As higher dimensional analogues 
of Riemann surfaces, the Riemann domains 
played a fundamental role in the early 
development of function theory of several 
complex variables (see \cite{BG60}, pp.
12-13). Such a space is given by a complex manifold 
$\,M\,$ together with a holomorphic map $\, p\,$ of 
$\,M\,$ into a domain in ${\mathbb C}^m\,$ such that 
each fiber of $\,p\,$ is discrete. This allows for 
the consideration of $m$-dimensional domains that
do not lie within ${\mathbb C}^m.\,$  

    It is well-known that on a Euclidean space 
harmonic functions are characterized by their 
mean-value properties over balls and spheres. 
In view of recent interest in non-smooth domains 
in analysis, it seems natural to consider similar 
properties for functions defined on a {\it 
semi-Riemann domain}, where singular 
or branch points as well as some 
non-discrete fibers may exist, thus 
allowing for possibly non-Stein 
parabolic spaces lying over a domain 
in ${\mathbb C}^m.\,$ In this paper 
{\it semi-harmonic functions} on a 
complex space are introduced. It is 
shown that for a continuous function 
on a Riemann domain, the 
"semi-harmonicity" is characterizable 
in terms of the local behaviour of 
the function such as the {\it solid}, 
{\it spherical}, as well as by the 
{\it near}, resp. {\it weak, harmonicity}. 
Furthermore, the Weyl's Lemma can be 
extended
\footnotemark\footnotetext{
The implication "(1) $\Rightarrow$ 
(2)" given in Theorem 4.2 of the 
above quoted paper is not correct 
as it stands unless under some 
further conditions on $\,\phi\,$ 
for instance, if either $\,\phi
\in C^0(X)\,$ or $\,\phi\,$ admits 
locally branchwise $L^1$-direct 
image $\,\hat\phi^{j}\in 
L^1_{{\rm loc}}(U')\,$ at each 
point of $\,X,\,$ for each branch 
of $\,U$ (under such circumstances 
the original proof remains valid).}%
: every locally forwardly $L^2$ 
weakly harmonic or semi-harmonic function 
on a Riemann domain is induced by a 
semi-harmonic function (Corollary 
\ref{C:removable sing. shf.}).

    For later applications to local and global 
characterizations of semi-harmonicity and 
holomorphicity of functions on a normal semi-Riemann 
domain (see \cite{cT07}), a class of generalized 
Euler vector fields is introduced (see \S 2 \& \S 5). 
The point of interest here lies in the fact that the 
Cauchy-Riemann, the $\bar \partial$-Euler, the $\bar 
\partial$-{\it Neumann} as well as the {\it radial} 
vector fields can be globally defined from a unified 
viewpoint. The relation between semi-harmonicity, 
{\it Dirichlet product}, and {\it harmonic residues} 
is also studied. Integral representation of the 
Bochner-Martinelli type and applications will be 
considered in a subsequent paper \cite{cT07}. 

     The author is indebted to the referee for
suggestions and corrections which led to the 
improvement of this paper.

\section{Preliminaries}\label{S:Pre.}  

     In what follows every complex space is assumed 
to be reduced and has a countable topology. The notions 
of $C^k$-differential forms, the exterior 
differentiation $d$, the operators $\,\partial, \> 
\bar \partial\,$ and $\,d^c : \> = (1/4\pi i)
(\partial - \bar \partial)\,$ are well-defined on 
complex spaces despite the presence of singularities 
(see \cite{cT79}, Chap. 4). Let $\,X\,$ be a complex 
space of dimension $\,m > 0\,$ and  $\,D \subseteq X\,$ 
an open subset. Denote by $\,C^\mu (D)\,$ the set of 
all $\,{\mathbb C}$-valued functions of class $C^\mu$ 
(when $\,\mu = \beta,\,$ locally bounded functions) 
on $\,D,\,$ and by $\,A^{k,\mu}(D)\,$ the set of 
$\,{\mathbb C}$-valued $\,k$-forms of class $C^\mu\,$ 
(when $\,\mu = \lambda,\;{\mathbb C}$-valued locally 
Lipschitz $k$-forms (\cite{cT79}, \S 4)) on $\,D.\,$  
The sets $\,C^\mu (\overline D)\,$ and $\,A^{k,\lambda} 
(\overline D)\,$ are similarly defined.

    Denote by $\,\|z\|\,$ the Euclidean norm of 
$\,z = (z_1, \cdots, z_m) \in {\mathbb C}^m,\,$ 
where $\,z_j = x_j + i\, y_j.\,$ Let the space 
$\,{\mathbb C}^m\,$ be oriented so that the form 
$\, {\upsilon}^m :\> = (dd^c \|z\|^2)^m\,$ is 
positive. Let $\,p : X \to {\mathbb C}^m\,$ be
a holomorphic map. Set $\,p^{[a]} :\> = p - p(a),
\; \forall a \in X.\,$ Clearly the form

\begin{equation}\label{E:Euc. form}
\upsilon_p : \> = dd^c \|p^{[a]}\|^2 = ({i \over
2\pi}) \partial \bar \partial \, \|p^{[a]}\|^2
\end{equation}
is non-negative and independent of  $\,a.\,$ 
Observe that the function $\,\|p^{[a]}\|\,$
satisfies the Monge-Amp\`ere equation:

\begin{equation}\label{E:Monge-Amp. equ.}
(dd^c \log \|p^{[a]}\|^2)^m \, = \,0.  
\end{equation}
Therefore the form

\begin{equation}\label{E:Poinc.form b}
\sigma_{a} : = {1 \over \|p^{[a]}\|^{2m}}\,
d^c \|p^{[a]} \|^2 \wedge {\upsilon}^{m-1}_p
\end{equation}
is $d$-closed.

  Let $\,d\upsilon_r,\,$ respectively,
$\,{d\sigma_r},\,$ denote the Euclidean volume
element of the $r$-ball ${\mathbb B}(r),\,$ 
respectively, the $(2m-1)$-sphere
$\,\mathbb{S}(r) = \mathbb{S}^{2m-1}(r),\,$ in
$\,\mathbb{C}^m.\,$ Set $\,d\upsilon_{[a],r} :\> 
= (p^{[a]})^*{d\upsilon_r},\; d\sigma_{[a],r} 
:\> = (p^{[a]})^*{d\sigma_r}.\,$ In particular, 
denoting by $\,t_{(- a')}\,$ the translation: 
$\,z\mapsto z - a'\,$ (where $\,a' \in \mathbb{C}^m
$), the forms $\,d\upsilon_{[a'],r} :\> = {t_{(- a')}}^*
{d\upsilon_r}, \; d\sigma_{[a'],r} :\> =
{t_{(-a')}}^*{d\sigma_r}\,$ are defined. For $\,a 
\in D\,$ and $\,r > 0,\,$ set $\,D_{[a]} (r) : \> = 
\{z \in D \,|\,\|p^{[a]}(z)\| < r\},\; D_{[a]} [r] : \> 
= \{z \in D\, |\,\|p^{[a]}(z)\| \le r\},\,$ and
$\,D (r) : \> = \{z \in D \, |\,\|p(z)\| < r\}.\,$ 
If $\,p =$ the identity map of $\,{\mathbb C}^m,\,$ 
write $\,\mathbb{B}_{[a']} (r) = D_{[a']}(r),\;
\mathbb{S}_{[a']}(r) = \partial {\mathbb B}_{[a']}
(r),\,$ and omit the subscript if $\,a' = 0.\,$ Set 
$\,|\,\mathbb{S}\,| = {\rm vol}\,
(\mathbb{S}(1))\,$ and $\,|\,{\mathbb B}\,| = 
{\rm vol}\,(\mathbb{B}(1))$.

    A complex space $\,X\,$ together with a
holomorphic map $\, p : X \to \Omega,\,$ where
$\,\Omega\,$ is a domain in $\,\mathbb{C}^m$, is 
called a {\it semi-Riemann domain} (over $\,\Omega$) 
iff there exists a thin analytic set $\,\Sigma\,$ in 
$\,\Omega\,$ for which the inverse image $\,\Sigma_p : 
= p^{-1}(\Sigma)\,$ is thin in $\,X\,$ and the 
restriction $\,p : X^0 :\> = X \,\backslash \,
\Sigma_p \to \Omega_0 :\> = \Omega \, \backslash 
\,\Sigma\,$ has discrete fibers. If $\,\Sigma = 
\emptyset, \,$ then $\,(X,p)\,$ is a {\it Riemann domain} 
in the sense of \cite{BG60}, p. 19 and \cite{GR79}, p.
135, (also \cite{FS87}, p. 116, where $\,X\,$ is
assumed a normal space). If $\,p : X \to \Omega\,$ is in
addition a local homeomorphism, then $\,(X,p)\,$
is said to be {\it unramified}. Every proper 
holomorphic map of a pure $m$-dimensional complex 
space into a domain $\,\Omega \subseteq \mathbb{C}^m\,$ 
of strict rank $\,m\,$ is a semi-Riemann domain 
(\cite{AS71}, p. 117).

    If $\,Y\,$ is an irreducible and locally 
irreducible space and $\,f : X \to Y\,$ is a light, 
proper, holomorphic map, then by the Andreotti-Stoll's 
theorem (\cite{wS66}, Lemma 2.3; \cite{AS71}, Lemma 2.2), 
the map $\,f : X \to Y\,$ is an analytic covering with 
sheet number $\,s = \deg\,(f)\,$ given by

\begin{equation}\label{E:sheet no.}
\deg\,(f) : = \,\sum \, \{\nu_f^y (z)\,|\, z \in f^{-1}(y)\}, 
\quad \forall  y \in Y.
\end{equation}
where $\,\nu_f^y (z)\,$ denotes the {\it multiplicity of 
$\,f\,$ at} $\,z\,$ (\cite{wS66}, p. 22).

   Unless otherwise mentioned, let $\,p = (p_1,
\cdots, p_m) : X \to \Omega\,$ be a semi-Riemann 
domain of dimension $\,m > 0.\,$ For each $\,a \in 
D^0 :\> = D \cap X^0,\,$ there exists an open 
neighborhood $\,N\,$ with closure in $\,D^0\,$ such 
that: i) $p^{-1}(y) \cap \overline N 
= \{a\},\,$ where $\,a': = p(a)$; ii) for a sufficiently 
small ball $\,U' = \mathbb{B}_{[a']}(\rho)\,$ in $\,
\mathbb{C}^m,\; U_a :\> = p^{-1}(U')\cap N = p^{-1}(U')
\cap \overline N\,$ is connected and the mapping $\,p 
\rfloor U_a : U_a \to U'\,$ is an analytic covering; 
iii) every branch $\,V^k,\; k = 1,\cdots, s_a,\,$ 
of $\,U_a\,$ contains $\,a$; and iv)

\begin{equation}\label{E:loc.deg.form.}
s_a = \,\deg \,(p \rfloor U_a) =
\,\nu_p^{a'}(a)
\end{equation}
(\cite{wS66}, Proposition 1.3). For convenience call
such $\,U_a\,$ a {\it pseudo-ball} ({\it of
radius} $\,\rho$) {\it at} $\,a.\,$  Let $\,X^*\,$ be 
the largest open subset of $\,X\,$ on which $\,p\,$ 
is locally biholomorphic, and set $\,D^* : \> = D \cap 
X^*.\,$

   Let $\,W \subseteq X^*\,$ be an open set and   
$\,T_a (W)\,$ the real tangent space at a point 
$\,a \in W.\,$ Denote by $\,(\;,\;)\,$ the Euclidean 
inner product (induced under $\,p$) on the tangent 
space $\,T_a (W),\; a \in W.\,$ It extends naturally 
to a bilinear form on the complexified tangent space 
$\,\mathbb{C}T_a (W).\,$ The gradient vector field 
$\,\nabla \phi\,$ of a $\,C^1$-function $\,\phi : W 
\to \mathbb{C}\,$ is then well-defined in the standard 
fashion; thus setting $\,\tilde x_j = {\rm Re}\>(p_j),
\; \tilde y_j = {\rm Im} \>(p_j),\; 1 \le j \le m,\,$
the partial derivatives

\[
\phi_{_{\tilde x_j}} =
\,{\partial \phi\over \partial \tilde x_j}
 : \> = (\nabla \phi, {\partial \over \partial \tilde x_j})
;\;\;\; \phi_{_{\tilde y_j}}  =
{\partial \phi \over \partial \tilde y_j} :
 \> = (\nabla \phi, {\partial \over \partial \tilde y_j}),
\]
are well-defined. The {\it Cauchy-Riemann}, respectively 
{\it anti-Cauchy-Riemann, vector fields}, are given 
by

\[
\bar \partial_k = {\partial \over 
\partial \bar p_{_k}} : = \,{1\over 2}\,{\nabla p_{_k}},
\quad {\rm respectively},\;\;\partial_k = {\partial 
\over \partial p_{_k}} : = \,{1\over 2}\, {\nabla 
\bar p_{_k}},
\] 
in $\,X^*\,$ for $\,k = 1,\cdots, m.\,$ More generally, 
if $\,g \in C^1 (D)\,$ define the associated $\,\partial$-, 
respectively, $\bar \partial$-, {\it Euler type vector 
field} 

\begin{equation}\label{E:Euler vec.field}
\begin{split}
{\mathcal E}_g\,= \, 2\,\sum_{k=1}^m \,{\, g_{
\bar p_k}\, \partial_k },\\
\bar {\mathcal E}_g\,= \, 2\,\sum_{k=1}^m \,{\, g_{
p_k}\, \bar \partial_k },  
\end{split}
\end{equation}
in $\,D^*,\,$ where $\,g_{_{\bar p_k}} : = \,\bar 
\partial_k g,\; g_{_{p_k}} : = \,\partial_k g.\,$ 
Observe that if $\,h \in {\mathcal
O}(D)\,$ (resp. $\,h \in \overline {\mathcal
O}(D)$), the set of all holomorphic (resp. 
anti-holomorphic) functions, then $\,\nabla h =
\bar {\mathcal E}_h \,$ (resp. $\,\nabla h
= {\mathcal E}_h$) is an Euler type vector field. 
To each continuous mapping $\, \xi = (\xi_1, \cdots, 
\xi_{2m}) : W \to \mathbb{C}^{2m}\,$ is associated a 
(complex) vector field

\[
\partial_{_\xi} \,: = \,\sum_{k=1}^m
{\big(\xi_{_{2k-1}} \frac {\partial}{\partial
\tilde x_k}\,+ \, \xi_{_{2k}} \frac {\partial}
{\partial \tilde y_k}\big) },
\]
which, for notational convenience, shall be identified 
with $\,\xi.\,$ It follows that

\begin{equation}\label{E:comp.dir.deri.}
\partial_{_{\nabla g}}\,= \, {\mathcal E}_g \,+\, 
\bar {\mathcal E}_g \,= \,\sum_{k=1}^m \nabla_k \,g,
\end{equation}
where

\begin{equation}\label{E:part.dir.deri.}
\nabla_k \,g \> := \, 2 \, (g_{_{p_k}}\, \bar
\partial_k \, +\, g_{_{\bar p_k}}
\, \partial_k)
\end{equation}
is the $k$-th {\it partial gradient vector field} 
of $\,g.\,$ The following Lemma, which gathers some 
useful identities, is easily established:

\begin{lemma}\label{L:Dir. and semi-Laplace oper.} 
{\rm (1)} If $\,\phi,\; \psi \in C^1 (D^*),\,$ 
then

\begin{equation}\label{E:Dir. and dir. deri.}
d \phi \wedge d^c \psi \wedge \upsilon_p^{m-1} 
= {1 \over 4m}\,\partial_{_{\nabla \phi}} 
(\psi)\,\upsilon_p^m.
\end{equation}
{\rm (2)} If $\,g \in C^2(D^*),\,$ then for each 
pseudo-ball $\,U \subseteq D^*,\,$ {\rm (i)}

\begin{equation}\label{E:semi-Laplace operator}
dd^c \,(g \,{\upsilon}_p^{m-1}) = {1\over 4m}
\,(\triangle_{{p_{_U}}} g)\,{\upsilon}_p^m,
\end{equation}
where $\,\triangle_{p_{_U}}$ denotes the 
$\,p_{_U}$-pull-back of the Laplace operator of the
Euclidean metric on $\,\mathbb{C}^m;\,$ {\rm (ii)}

\begin{equation}\label{E:L.C.R.equa.}
\partial_{_{\nabla g}} (\phi) \,
=\,{1\over 2}\,(\triangle_{p_{_U}} (g\,\phi) \,
- \, g \,\triangle_{p_{_U}} \phi\,- \,\phi\,
\triangle_{p_{_U}} g), \quad \forall \phi \in
C^2 (U).
\end{equation}
\end{lemma}

\section{Integral averages}\label{S:Int.aver.} 

     Denote by $\,dD\,$ the (maximal) {\it boundary 
manifold} of $\,{\mathcal R}(D)\,$ in $\,{\mathcal
R}(X),\,$ the manifold of simple points of $\,X,\,$ 
oriented to the exterior of $\,{\mathcal R}(D)$ 
(\cite{cT79}, p. 218). If $\,\overline U\,$ is 
compact, $\,a \in U^0,\,$ and $\,\phi \in C^0 
(U_{[a]} [r_0]),\,$ define the {\it solid}, 
resp. {\it spherical, mean-value function}
of $\,\phi\,$ (with resp. to $\,p^{[a]}$) by

\begin{equation}\label{E:solid m.v.}
\langle \phi \rfloor U \rangle_{a,r}: \> = {1 \over
r^{2m}} \mathop{\int}_{U_{[a]}(r)}{\phi \,
\upsilon^m_p}, \quad  \forall r \in
(0,r_0), 
\end{equation}

\begin{equation}\label{E:spher.m.v.}
[\phi \rfloor U]_{a,r} : \> =
\mathop{\int}_{dU_{[a]}(r)}{\phi
\,\sigma_{a}}, \quad  \forall r \in
(0,r_0). 
\end{equation}

   Let $\,\phi \in C^0 (D)$ and $a \in D^0.\,$
Then: $\phi \,$ is said to have (1) the (local) 
{\it solid mean-value property at} $\,a\,$ iff 
there exists a pseudo-ball $\,U \subseteq D\,$ 
at $\,a\,$ of radius $\,r_0 > 0\,$ such that

\begin{equation}\label{E:solid mean.}
\langle \phi \,\rfloor U \rangle_{a,r} =
\nu_p(a) \,\phi (a) \quad \forall r \in (0,r_0).
\end{equation}
(2) the (local) {\it spherical mean-value property 
at} $\,a\,$ iff there exists a pseudo-ball $\,U 
\subseteq D\,$ at $\,a\,$ of radius $\,r_0 > 0\,$ 
such that

\begin{equation}\label{E:spher. mean.}
[\phi \rfloor U]_{a,r} =  \nu_p(a) \,\phi
(a), \quad \forall r \in (0,r_0). 
\end{equation}

     Given a pseudoball $\,U\,$ in 
$\,X,\,$ choose a $C^\infty$-partition 
of unity $\,\{(U^j_\nu,\rho^j_\nu)\}
\,$ in terms of the local (covering) 
sheets $\,U^j_\nu\,\subseteq U_{*}\,$ 
contained in a branch $\,V^j.\,$ 
Given $\,\phi\in L^1_{{\rm loc}}
(U)$ (respectively, $\,C^0(U)$), 
set $\,\phi^j_{_\nu} := \rho^j_\nu
\phi\,$ for each pair $\,(j,\nu);\,$ 
there exists an induced function 
$\,\hat\phi^j_{_\nu}\in L^1_{{\rm loc}}
((U^j_\nu)')$ (respectively, 
$\,C^0((U^j_\nu)')$) such that $\,
p^*\hat\phi^j_{_\nu} = \phi^j_{_\nu}
\,$ on $\,U^j_\nu.\,$ The following 
Lemma shows that the local solid 
and the spherical mean-value 
properties of $\,\phi\,$ are 
equivalent:

\begin{lemma}\label{L:const. aver.} 
   Let $\,U\,$ be a pseudo-ball at 
a point $\,a\in X^0\,$ of radius $\,r_0$. 
Then for each $\,\phi\in  C^0(\overline U)$, 
the following are equivalent:

{\rm (a)} $\,\langle \phi \rfloor U \rangle_{a,r}
= {\rm const.} = A, \; \forall r \in (0,r_0)$;

{\rm (b)} $\,[\phi \rfloor U]_{a,r} =
{\rm const.} = A, \; \forall r \in (0,r_0)$.
\end{lemma}

\begin{proof} 
Observe that, if $\hat \phi \in C^0 
(\mathbb{B}_{[a']}[r_0]),\,$ then

\begin{equation}\label{E:ball and sphere}
\mathop{\int}_{\mathbb{B}_{[a']}(r)}
{\hat\phi\,d \upsilon_{[a'],r}} = \int_0^r
\big(\mathop{\int}_{\mathbb{S}_{[a']}(t)}{\hat
\phi \,d\sigma_{[a'],t}}\big)\, dt, \quad 0 
< r < r_0. 
\end{equation}
Assume at first that $\,
[\phi \rfloor U]_{a,t} 
= A,\;\forall t\in (0,r_0)$. 
Denoting by $\,{\mathfrak s}_a\,$ 
the number of irreducible branches 
of $\,U,\,$ one has for such 
$\,t,\,$ 

\[
  \begin{split}
  A\,|\,\mathbb{S}\,|\,t^{2m-1} =
  \mathop{\int}_{dU_{[a]}(t) \backslash T}{\phi
  \,d \sigma_{[a],t}}
  & = \,\sum_{k=1}^{{\mathfrak s}_a}{\sum_{\nu}}'
  \mathop{\int}_{dV^k_{[a]}(t)\cap U^k_\nu\backslash T}
  {\phi^k_{_\nu}\,d\sigma_{[a],t}} \\
  & = \,\sum_{k=1}^{{\mathfrak s}_a}{\sum_{\nu}}'
  \mathop{\int}_{\mathbb{S}_{[a']}(t)\cap (U^k_\nu)'}
  {\hat\phi^k_{_\nu}\,d\sigma_{[a'],t}}, 
  \end{split}
\]
where $\,{\sum_{\nu}}'\,$ denotes
the limit of a sum of integrals
(or functions) over the indices $\nu$ 
of an (increasing) covering of $\,K_n
\cap V^k\,$ by the $\,U^k_\nu,\,
\{K_n\}\,$ being a (strictly 
increasing) exhausting sequence of 
compact subsets of $\,U^{*}.\,$ 
Thus, integrating the above relation 
over $(0,r)$, $\,0 < r < r_0,\,$ 
yields

\[
{r^{2m}\over {2m}}\,|\,\mathbb{S}\,| A =
\,\sum_{k=1}^{{\mathfrak s}_a}{\sum_{\nu}}'
\mathop{\int}_{\mathbb{B}_{[a']}(r)
\cap (U^k_\nu)'}{\hat\phi^k_{_\nu}
\,d\upsilon_{[a'],r}} = \,\sum_{k=1}^{{\mathfrak s}_a} 
\mathop{\int}_{{V^k_{[a]}(r)\backslash T}}
{\phi\,d\upsilon_{[a],r}}.
\]
Therefore

\[
A = {1 \over r^{2m}} 
\mathop{\int}_{U_{[a]}(r)}
{\phi\,{\upsilon_{p^{[a]}}}^m} = \langle\phi
\rfloor U \rangle_{a,r}, \quad 0 < r < r_0.
\]
Similarly, if $\,\langle \phi \rfloor
U \rangle_{a,r} = A, \; \forall r \in (0,r_0)$,
then $[\phi \rfloor U]_{a,r} = A, \;
\forall r \in (0,r_0)$.
\end{proof}

     Of importance to harmonic function theory 
is the {\it Dirichlet product}, which, on a 
semi-Riemann domain, can be defined as follows: 
if $\,\eta,\;\phi : D \to \mathbb{C}\,$ are locally 
Lipschitz functions (\cite{cT79}, \S 4), set

\begin{equation}\label{E:Dr.pr.}
\lbrack \eta,\phi \rbrack_{_D} \,: =
\mathop{\int}_{_D}{d \eta \wedge d^c \bar \phi 
\wedge\upsilon_p^{m-1}},
\end{equation}
provided the integral exists. (For further properties 
and applications of this product, see \cite{cT07}). 
The definition \eqref{E:Poinc.form b} and the Stokes' 
theorem (\cite{cT79}, (7.1.3)) imply the following

\begin{lemma}\label{L:relat. between solid and spher. aver.} 
Let $\,\eta ,\,\phi \in C^1 (D)\,$ and $\,a\ 
\in D^0.\,$ Then for every neighborhood $\,U\,$ 
of $\,a\,$ and $\,r_0 > 0\,$ such that $\,U_{[a]}
(r_0) \subset \subset D$,

\begin{equation}\label{E:rel.spher. and solid MV}
[\phi \rfloor U]_{a,r} = \,
\langle \phi \rfloor U \rangle_{a,r} +
{1 \over r^{2m}}\, [\phi, \|p^{[a]}\|^2]_{a,r},\quad 
\forall r \in (0, r_0).
\end{equation}
\end{lemma}

   A function $\,\phi \in C^0 (D)$ is called (locally) 
{\it nearly harmonic at} $\,a \in D^0\,$ iff there exists 
a pseudo-ball $\,U \subseteq D\,$ at $\,a\,$ of radius 
$\,r_0 > 0\,$ such that

\begin{equation}\label{E:near harm}
[\phi \rfloor U]_{a,r} =
\langle \phi \rfloor U \rangle_{a,r}, \quad
\forall r \in (0,r_0). 
\end{equation}

\begin{lemma}\label{L:near harm.} 
A function $\,\phi \in C^0 (D)$ is 
nearly harmonic at point $\,a\in D^0
\,$ iff there exists a pseudo-ball 
$\,U\subseteq D\,$ at $\,a\,$ of 
radius $\,r^* > 0$ such that

\[
[\phi \rfloor U]_{a,r} = {\rm
const}., \quad \forall r \in(0, r^*).
\]
\end{lemma}

\begin{proof} 
Assume $\,\phi \,$ is nearly harmonic 
at point $\,a \in D^0.\,$ It suffices 
to consider the case where $\,\phi\,$ 
is real-valued. In terms of the 
Euclidean volume elements the 
condition \eqref{E:near harm} can be 
written

\[
{\sum_{\nu}}'\mathop{\int}_{U_{[a]}(r)}
{\phi^k_{_\nu}\,d\upsilon_{[a],r}} 
= {r \over 2m}{\sum_{\nu}}'
\mathop{\int}_{dU_{[a]}(r)}{\phi^k_{_\nu}
\,d\sigma_{[a],r}}.
\]
Hence by the formula \eqref{E:ball 
and sphere},

\[
{\sum_{\nu}}'\mathop{\int}_{dU_{[a]}(r)}
{\phi^k_{_\nu}\,d\sigma_{[a],r}} = 
{\sum_{\nu}}'{d \over dr}\,
\big({r \over 2m}\mathop{\int}_{dU_{[a]}(r)}
{\phi^k_{_\nu}\,d\sigma_{[a],r}}\big).
\]
Thus 

\[
{d \over dr}{\sum_{\nu}}'
\big(\mathop{\int}_{dU_{[a]}(r)}
{\phi^k_{_\nu}\,d\sigma_{[a],r}} \big) 
= {(2m-1) \over r}{\sum_{\nu}}'
\mathop{\int}_{dU_{[a]}(r)}
{\phi^k_{_\nu}\,d\sigma_{[a],r}}.
\]
Hence one has

\[
{d \over dr}\big(\mathop{\int}_{dU_{[a]}(r)}
{\phi\,d\sigma_{[a],r}} \big) 
= {(2m-1) \over r}\mathop{\int}_{dU_{[a]}(r)}
{\phi\,d\sigma_{[a],r}}.
\]
It follows that for some $\,r^* > 0,\,$

\[
{1 \over {|\,\mathbb{S}\,|\,r^{2m-1}}} 
\> \|\mathop{\int}_{dU_{[a]}(r)}
{\phi\,d\sigma_{[a],r}}\|\, 
=\,{\rm const.},\quad \forall r\in (0,r^*).
\]
Therefore the function
$\,\|[\phi \rfloor U]_{a,r}\|,\,$ hence also 
$\,[\phi \rfloor U]_{a,r},\,$ is constant for
such $\,r.\,$ The converse assertion is an 
immediate consequence of Lemma 
\ref{L:const. aver.}.
\end{proof}

\begin{proposition}\label{P:deg. form.} 
If $\,U \subseteq X\,$ is a pseudo-ball at
$\, a \in X^0\,$ and $\,\phi \in C^0 
(\overline U)\,$ such that $\,\phi \rfloor\, 
U\cap p^{-1} (p(a)) = {\rm const.},$ then

\begin{equation}\label{E:deg. form.}
\deg \,(p\rfloor U)\,\phi (a) \, = \, \lim_{r \to 0}
\> [\phi \rfloor U]_{a,r} = \, \lim_{r\to 0}\,
\langle \phi \rfloor U \rangle _{a,r}. 
\end{equation}
\end{proposition}

\begin{proof} 
Without loss of generality assume $\,U = X,\,$ 
and let $\,\tau: = \|p^{[a]}\|^2 : X \to [0, 
\infty).\,$ The identity \eqref{E:rel.spher. and 
solid MV} implies that $\,[1]_{a,r} 
= \,\langle 1\rangle_{a,r},\;\forall r > 0.\,$ 
Also, it follows from \cite{cT79}, Proposition 
5.2.2 and the sheet number formula 
\eqref{E:sheet no.} that

\begin{equation}\label{E:deg.form.a}
\langle 1 \rangle_{a,r} = \,\deg \,(p)
\,\langle 1 \rfloor \Omega
\rangle_{a',r} = \,\deg \,(p),
\end{equation}
where the integral $\,\langle 1 \rfloor \Omega 
\rangle_{a',r} = 1\,$ is defined in terms of the 
Euclidean volume element $\,\upsilon^m\,$
on $\,\mathbb{C}^m.\,$ Set $\,\tilde \phi :\> = \phi -
\phi(a).\,$ For each $\,\varepsilon > 0,\,$ let
$\,W_{\varepsilon}^\alpha\,$ be a neighborhood of
$\,z_\alpha \in \tau^{-1}(0)\,$ such that
$\,\|\tilde \phi (z)\| < \varepsilon, \;
\forall z \in W_{\varepsilon}^\alpha.\,$ Take 
$\, r_0 > 0\,$ and choose an open covering of 
$\,p^{-1}(a') \cap X_{[a]} [r_0]\,$ by the open 
sets $\,W_{\varepsilon}^\alpha$. Then there exists
a constant $\,\delta > 0\,$ such that $\,\tau^{-1}
([-\delta, \delta]) \cap X_{[a]} [r_0] \subseteq
W_\varepsilon := \cup_{\alpha} \,
W_{\varepsilon}^\alpha\,$ (\cite{cT79}, (1.1.5)). 
This implies that $\,\sup_{X_{[a]}[\delta]} 
\|\tilde \phi\|\le \varepsilon.\,$ \ Hence

\begin{equation}\label{E:deg.est.}
\| \langle \tilde \phi \rangle_{a,r}\| 
\le \sup_{X_{[a]}[\delta]} \|\tilde \phi\| \,
\langle 1 \rangle_{a,r} \,\le \,\varepsilon 
\, \deg \,(p), \quad \forall r \in (0, \delta].
\end{equation}
Therefore it follows from the relations 
\eqref{E:deg.form.a} and \eqref{E:deg.est.} 
that

\[
 \lim_{r \to 0} \, \langle \phi \rfloor U 
\rangle_{a,r} = \,\phi (a) \,\deg \,(p). 
\]
The remaining assertion on the spherical mean-value 
is similarly proved.
\end{proof}

\begin{proposition}\label{P:max.prin.} 
{\rm (}Maximum principle{\rm )} Let $\,D\,$ be a domain 
in $\,X.\,$ Assume: {\rm i)} either $\, D \subseteq X^0
\,$ or $\,D\,$ is irreducible; {\rm ii)} $\phi : D 
\to [- \infty, \infty)\,$ is upper-semicontinous and 
bounded above; {\rm iii)} $\forall a \in D^0,\,$ there 
exists a neighborhood $\,U \subseteq D\,$ of $\,a\,$ 
such that

\begin{equation}\label{E:sub-mean-value}
\nu_p(a) \, \phi (a) \le \langle
\phi \,\rfloor U \rangle_{a,r} \quad \forall r \in (0,r_0).
\end{equation}
Then $\,\phi \,$ satisfies the maximum
principle on $\,D${\rm :} if for some $\,z_0 \in D^0$,
$\phi (z_0) = K : \, = \sup\, \{ \phi (z)\,| \,z
\in D\} \not \equiv -\infty,\,$ then $\,\phi =$ constant.
\end{proposition}

\begin{proof}
   Let $\,M = \{z \in D\,| \,\phi (z)
< K \}.\,$ For any $\,a_0 \in D^0 \backslash M,\,$ 
choose a neighborhood $\,U \subseteq D^0\,$
such that the inequality \eqref{E:sub-mean-value} 
holds. Without loss of generality, assume that $\,U\,$ 
is a pseudo-ball at $\,a_0\,$ of radius $\,\rho.\,$ 
Suppose $\,M \cap U_{[a_0]}(\rho) \not = \emptyset.\,$ 
Since $\,\phi (z) < K\,$ for all $\,z\,$ in a 
neighborhood of each $\,z^*\in M \cap
U_{[a_0]}(\rho)$, it follows from \cite{cT79},
Proposition 5.2.2, and \eqref{E:sheet no.} that

\[
  \begin{split} 
\rho^{2m} \,\nu_p(a_0)\,\phi (a_0) & \le 
\mathop{\int}_{U_{[a_0]}(\rho)}{\phi
\,\upsilon^m_p} \\
& < \mathop{\int}_{U_{[a_0]}(\rho)}{K
\upsilon^m_p} = \,\rho^{2m} \,\deg \,(p\rfloor U)\,K.
  \end{split}
\]
Hence by the relation \eqref{E:loc.deg.form.},
this implies that $\,\phi (a_0) < K,\,$ a
contradiction. Therefore $\,M \cap U_{[a_0]}(r) = 
\emptyset.\,$ Thus the set $\,D^0 \backslash M\,$ 
is open and non-void. It follows from the connectedness 
of $\,D^0\,$ that $\,M \cap D^0 = \emptyset.\,$ 
Consequently $\,\phi (z) = K\,$ in $\,D^0$, hence 
also in $\,D$.
\end{proof}

\section{Semi-harmonicity}\label{S:Semi-harm.}

     For later use set $\,C^\lambda 
(\overline D) : = A^{0,\lambda} 
(\overline D),\,$ $\,C^{1,1} 
(\overline D) : = \{\phi \in C^1 
(\overline D) \,|\,\phi_{_{\tilde x_j}}
\,{\rm and} \;\phi_{_{\tilde y_j}} 
\in C^\lambda (\overline D),\,1\le 
j\le m\},\,$ and $\,C^{\lambda,(c)}
(\overline D)\,:= \,\{\eta \in C^1 
(\overline D) \,|\, \eta\,\rfloor 
\partial D = {\rm const.}\}.\,$ 
Denote by $\,C^\mu_0 (D)\,$ the set 
of compactly supported $C^\mu$-functions. 
The sets $\, C^\lambda (\partial D)
,\,$ and $\,C^{1,1} (D),\; C^{1,1}_0 
(D)\,$ are similarly defined. Let 
$\,\rho \in A^{2m,0}(D).\,$ A function 
$\,f\,$ on $\,X\,$ is said to be {\it 
locally integrable} ($f\in L^1_{{\rm loc}}
(X)$) provided so is every $2m$-form 
$\,f\chi\,$ with $\,\chi\in A^{2m,0}
(X).\,$ Similarly define $\,
L^2_{{\rm loc}}(X)\,$ (with $\,|g|^2
\chi\,$ in place of $\,g\chi$). A 
{\it weak solution} of the semi-Poisson 
equation

\begin{equation}\label{E:Poisson equ.}
dd^c \,(\phi \,{\upsilon}_p^{m-1})
= \,\rho 
\end{equation}
is a locally integrable function $\,\phi: D \to 
\mathbb{C}\,$ such that $\forall \> C^2$-function 
$\,u : D \to [0, \infty)\,$ with compact support 
in $\,D^*$, 

\[
\mathop{\int}_{D}{\phi \,dd^c u \wedge
\upsilon_p^{m-1}} = \mathop{\int}_{D}{\rho\, u \,
\upsilon_p^m }. 
\]
A ({\it strong}) {\it solution} of the equation 
\eqref{E:Poisson equ.} is a $C^2$-function $\,\phi: D 
\backslash A \to \mathbb{C}\,$ for some thin analytic 
subset $\,A\,$ of $\,D\,$ such that $\,\phi \,$ 
satisfies the equation \eqref{E:Poisson equ.} (pointwise)  
in $\,D \backslash \,A.\,$ The Stokes theorem shows that 
every strong solution in $\,C^{1,1} (\overline D)\,$ of 
the equation \eqref{E:Poisson equ.} is a weak solution.
It will be shown that on a Riemann domain, if $\,\partial 
D$ is reasonably smooth, a bounded weak solution in $\,
C^0(\overline D)\,$ of the semi-Poisson equation depends 
continuously on its boundary values on $\,\partial D.\,$

      A locally integrable function $\,\phi: D 
\to \mathbb{C}\,$ is called: (1) {\it weakly harmonic 
in} $\,D\,$ (with respect to $\,p$) iff $\,\phi\,$ is 
a weak solution of the semi-Laplace equation

\begin{equation}\label{E:semi-Laplace equ.}
dd^c \,(\phi \, {\upsilon}_p^{m-1}) = 0  
\end{equation}
in $\,D;\,$ (2) {\it semi-harmonic in} $\,D$ 
(with respect to $\,p$) iff there exists a thin 
analytic subset $\,A\,$ of $\,D\,$ such that 
$\,\phi\,$ is a $C^2$-solution of the semi-Laplace 
equation \eqref{E:semi-Laplace equ.} in $\,D 
\backslash \,A.\,$ Also, $\,\phi: D \to 
\mathbb{C}\,$ is called {\it semi-harmonic at} 
$\,a \in D\,$ iff $\,\phi\,$ is semi-harmonic in 
a neighborhood of $\,a.\,$ The real and imaginary 
parts of every weakly holomorphic or pluri-harmonic 
function in $\,D\,$ are semi-harmonic. Locally 
every pure $m$-dimensional complex space $\,X\,$ is 
an analytic covering of a domain in $\,\mathbb{C}^m.\,$ 
Thus by means of the identity 
\eqref{E:semi-Laplace operator}, the 
{\it semi-harmonicity} of a function $\,\phi\,$ on a
domain $\, D \subseteq X\,$ is intrinsically 
defined by the requirement \eqref{E:semi-Laplace equ.} 
in terms of the local covering maps of $\,X.\,$ Let $\,
{\mathfrak H}(D),\,$ resp. ${\mathfrak H}_w(D),
\,$ denote the set of all semi-harmonic, resp. weakly 
harmonic, functions in $\,D.\,$

\begin{theorem}\label{T:charac. semi-harm.}
Let $\,(X, p)\,$ be a Riemann domain and 
$\,\phi\in C^0 (X)\,$. The following 
conditions are equivalent: 

   {\rm (1)} $\,\phi \,$ is locally nearly harmonic 
in $\,X$.

  {\rm (2)} $\, \phi \,$ has the local spherical
mean-value property in $\,X$.

   {\rm (3)} $\, \phi \,$ has the local solid
mean-value property in $\,X$.
\end{theorem}

\begin{proof} 
By Lemma \ref{L:near harm.}, the 
near-harmonicity at $\,a \in X\,$ 
of $\,\phi \,$ is equivalent to
the constancy of its local spherical 
mean-value at $\,a$. Hence it follows 
then from Lemma \ref{L:const. aver.}, 
Proposition \ref{P:deg. form.} and 
the relation \eqref{E:loc.deg.form.} 
that at each point of $\,X,\,$ the 
above three conditions on $\,\phi\,$ 
are equivalent.
\end{proof}

     An element $\phi\in 
L^1_{{\rm loc}}(X)$ is said to 
be {\it locally forwardly 
square-integrable} ($\phi\in 
L^2_{{\rm loc}}[X]$) if there 
exists a pseudoball $\,U\Subset 
X\,$ at each point of $\,X\,$ 
such that the following integral
exists: $\,\int_{U'\backslash 
\Delta'}{\|(\sum_{\nu}'
\hat\phi^{j}_{\nu})(z')\|^2
\,dv(z')} < \infty,\,$ where the 
(measurable) function 
$\,\sum_{\nu}'\hat\phi^{j}_{\nu}  
: U'\backslash\Delta'\to {\mathbb C}
\,$ is defined by: $\,z'\mapsto 
\lim_{n\to \infty}\sum_{l=1}^{N^j_n}
\hat\phi^j_{n_l}(z'),\,$ the 
summation being taken over a 
covering of $\,K_n\cap V^j\,$ 
by the open sets $\,U^j_{n_l},
\,l = 1,\cdots, N^j_n,\,$ the 
$\,\{K_n\}\,$ being an exhausting 
sequence of compact subsets of 
$\,U^*.\,$ Note that it can be 
shown (by using \cite[Theorem 5.2.2]
{cT79}) that the function 
$\,\sum_{\nu}'\hat\phi^{j}_{\nu}
\,$ is integrable on $\,U'
\backslash\Delta'.\,$ The 
following characterization of 
semi-harmonicity also gives a 
criterion for the removability 
(in a weak sense) of analytic 
singularities:

\begin{theorem}\label{T:weak harm. and semi-harm.}
Let $\,(X, p)\,$ be a Riemann domain. 
The following assertions "{\rm (2)} 
$\Rightarrow$ {\rm (3)}  $\Rightarrow$ 
{\rm (1)}" hold; moreover, the 
implication "{\rm (1)} $\Rightarrow$ 
{\rm (2)}" is valid if $\,\phi\in 
L^2_{{\rm loc}}[X]${\rm :} 

   {\rm (1)} $\,\phi\in {\mathfrak H}_w
(X).\,$  

   {\rm (2)} $\phi\,$ is locally 
integrable in $\,X\,$ and defines 
a current $\,[\phi]\,$ induced
by a function $\,\tilde\phi\in 
C^{\beta\cap {\mathfrak m}}(X)
\cap C^0(X^*)\,$ satisfying the 
local solid mean-value property 
in the domain of continuity of 
$\,\tilde\phi.\,$

   {\rm (3)} $\,\phi\in 
{\mathfrak H}(X).\,$
\end{theorem}

\begin{proof} 
    To prove the assertion "(1) 
$\Rightarrow$ (2)", let $\,U 
\subseteq X\,$ be a pseudo-ball 
of radius $\,r_0\,$ at a point 
$\,a \in X.\,$ Choose a non-negative 
function $\,\alpha \in C^\infty_0 
(\mathbb{R}^{2m})\,$ with support 
contained in $\,\mathbb{B}[1]\,$ 
such that $\,\int_{\mathbb{R}^{2m}} 
\alpha \,d\upsilon = 1.\,$ Set 
$\,\alpha_\varepsilon (x) :\> =
\varepsilon^{-2m}\, 
\alpha({x \over \varepsilon}),\; 
\alpha_\varepsilon^{(0)}(x) = 
\alpha_\varepsilon (- x),\; \forall 
\varepsilon > 0.\,$ Define $\,\hat
\phi_{j,\varepsilon} = 
\alpha_\varepsilon * (\phi| V^j)\,$ 
for each branch $\,V^j,\; 1 \le j
\le {\mathfrak s}_a,\,$ of $\,U,\,$ 
by setting

\[
\hat\phi_{j,\varepsilon} (x) := \int_{V^j}
{\alpha_\varepsilon (x-y')\,\phi(y)
\,d\tilde\upsilon}
\]
on $\,W_\varepsilon = \{ x \in U'\,|\, 
{\rm dist} (x, \partial U') > 
\varepsilon\},\,d\tilde\upsilon 
:= p^*({d\upsilon})\,$ being the 
pullback of the Euclidean volume 
element on $\,{\mathbb C}^m.\,$
Then $\,\hat \phi_{j,\varepsilon}
\in C^\infty(W_\varepsilon).\,$ 
Let $\,U''\subset\subset W\subset 
\subset p(V^j)\,$ be open subsets 
of $\,U'.\,$ Set $\,\varepsilon_1 
= {1 \over 2} {\rm dist} 
(\overline {U''}, \partial W).\,$ 
Then $\forall w\in C^\infty_0 (Q),
\,Q = U^j_\nu\,$ being a local 
sheet of $\,p^{-1}(U'')\cap 
V^j$ (at a point $\,z$), and 
denoting by $\,\hat w\,$ 
(respectively, $\,\hat\phi^{j,\nu}$) 
the function on $\,Q'\,$ induced 
by $\,w$ (respectively, $\,\phi|Q$),
one has $\,{\rm Spt}
(\alpha_\varepsilon * \hat w) 
\subseteq W,\;\forall\varepsilon
\in (0,\varepsilon_1)\,$ and

\begin{equation}\label{E:weak Laplacian.a}
\begin{split} 
\int_{Q'}{\hat\phi_{j,\varepsilon}\,
dd^c \hat w\wedge \upsilon_{[z']}^{m-1}}
& = \,{\rm const.}\,{(\hat\phi^{j,\nu},
\,\alpha_\varepsilon^{(0)} * 
\triangle \hat w)_{_{Q'}}}\\ 
& = \,\,{\rm const.}\,{(\hat\phi^{j,\nu},
\,\triangle (\alpha_\varepsilon^{(0)} * 
\hat w))_{_{Q'}}}.
\end{split}
\end{equation}
Therefore

\[
\int_{Q'}{\hat \phi_{j,\varepsilon}\,dd^c
\hat w \wedge \upsilon_{[z']}^{m-1}} \,= \,
 \,\int_{U''}{\hat\phi^{j,\nu} \,
 dd^c (\alpha_\varepsilon^{(0)} * \hat w)\, 
\upsilon_{[z']}^{m-1}}.
\]
Thus the weak harmonicity of 
$\,\hat\phi^{j,\nu}\,$ implies 
that each $\,\hat
\phi_{j,\varepsilon}\,$ is 
weakly harmonic in $\,V^j\,$ for 
sufficiently small $\,\varepsilon 
> 0.\,$ Hence for any $\,\hat u\in
C^2_0(U'')\,$ with $\,\hat u\ge 0$ 
and $\,\hat u = 1\,$ in a neighborhood 
of a point $\,z'\in U''_\varepsilon
\backslash \Delta',\,$ the relation

\begin{equation}\label{E:weak Laplacian.b}
\int_{U''}{\hat \phi_{j,\varepsilon}\,dd^c
\hat u \wedge \upsilon_{[z']}^{m-1}} 
= \int_{U''}{\hat u \,dd^c
\hat \phi_{j,\varepsilon} \wedge
\upsilon_{[z']}^{m-1}} = 0
\end{equation}
(for sufficiently small 
$\,\varepsilon > 0$) implies 
that the function $\,\hat 
\phi_{j,\varepsilon}\,$ is 
harmonic in $\,U''.\,$ Observe 
that the Dirichlet product

\[
[\hat \phi_{j,\varepsilon},
\|z' -a'\|^2]_{a',r} = \mathop{\int}_{U'_{a'}(r)}
{(r^2 - \|z' - a'\|^2)\,dd^c \hat\phi_{j,\varepsilon} 
\wedge \upsilon_{[a']}^{m-1}}  
\]
for small $\,r > 0.\,$ Hence the 
identity \eqref{E:rel.spher. and 
solid MV} implies that $\,\hat 
\phi_{j,\varepsilon}\,$ is nearly 
harmonic at $\,a'.\,$ Thus it follows 
from Lemma \ref{L:near harm.} and 
the relation \eqref{E:deg. form.} 
that there exists $\,r_{a'}\in 
(0,r_0)\,$ such that

\[
\hat\phi_{j,\varepsilon} (a') = \langle
\hat\phi_{j,\varepsilon}\,\rfloor
{\mathbb{B}_{[a']}}(r) \rangle_{a',r}, 
\quad\forall r \in (0,r_{a'}). 
\]
Here the $\,r_{a'}\,$ may be 
chosen to be independent of 
$\,\varepsilon\,$ (for small 
$\,\varepsilon\,$), since the 
open set $\,W_\varepsilon\,$ 
increases with decreasing 
$\,\varepsilon.\,$ Assume at 
first that $\,\phi\in C^0 (X).
\,$ Letting $\,\varepsilon 
\to 0\,$ this relation yields

\[
\lim_{\varepsilon\to 0}\,
\langle\hat\phi_{j,\varepsilon}\,\rfloor
{\mathbb{B}_{[a']}}(r) \rangle_{a',r}
= \lim_{\varepsilon\to 0} 
\hat\phi_{j,\varepsilon} (a') = s_j\phi(a), 
\quad \forall r \in (0,r_{a'}),
\]
the $\,s_j\,$ being the sheet number 
of $\,p|V^j$ (where the last equality 
follows from \cite[Theorem 5.2.2]
{cT79}). Suppose that $\,f\in 
C^{\infty}_0(U^{j}_{\nu})\,$ and 
$\,r\in (0,r_{a'}).\,$ There exists 
an $\,\varepsilon_r > 0\,$ such that 
$\,U'_{[y']}(\varepsilon)\subseteq 
U'_{[a']}(r)\,$ for all $\,y\in 
V^j_{[a]}(r)\cap {\rm Spt}(f).\,$ 
Then for all $\,\varepsilon \in 
(0,\varepsilon_r),\,$ 

\[	
  \begin{split} 
\langle f\,\rfloor {\bf 1}_{V^j}\rangle_{a,r} 
& = {1 \over r^{2m}} \mathop{\int}_{V^j_{[a]}(r)\cap U^j_\nu}
{f(y)\cdot (\int_{z'\in U'\backslash\Delta'}
{\alpha_{\varepsilon}^{(0)}(y'-z')\,d\upsilon(z')})
\,\upsilon_{p}^{m}}(y)\\
& = {1 \over r^{2m}}  \mathop{\int}_{(U^j_\nu)_{[a]}(r)}
{(\int_{V^j_{[a]}(r)\cap U^j_\nu}
{(\alpha_{\varepsilon}^0 \circ p^{[z]})(y)
f\,(y)\,d\tilde\upsilon}(y))\,\upsilon^m_p(z)}\\
& = \langle f_{_{\{j,\varepsilon\}}}
\,\rfloor {\bf 1}_{{\mathbb B}_{[a']}(r)}\rangle_{a',r}.
\end{split}
\]
Denoting by $\,
\psi^j_{_{\nu,\varepsilon}}\,$ 
the convolution $\,\alpha_\varepsilon 
* (\rho^{j}_{\nu}\phi),\,$ one has $\,
\hat\phi_{j,\varepsilon} = {\sum_{\nu}}'
\psi^j_{_{\nu,\varepsilon}},\,$ so 
that

\[  
\begin{split} 	
\langle\phi\,\lfloor {\bf 1}_{V^j}\rangle_{a,r} 
& = \sum_{\nu}\langle \psi^j_{_{\nu,\varepsilon}}\,
\rfloor {\mathbb B}_{[a']}(r)\rangle_{a',r}
= \langle \hat\phi_{j,\varepsilon}\,
\rfloor {\mathbb B}_{[a']}(r)\rangle_{a',r}. 
\end{split}
\]
Thus

\[
\langle \phi\,\rfloor U\rangle_{a,r} 
= \sum_{j}\langle\phi\,\lfloor 
{\bf 1}_{V^j}\rangle_{a,r} = 
\sum_{j}s_j \phi(a) = \nu_p(a)\,\phi(a), 
\quad \forall r \in (0,r_{a'}).
\]
Hence the local solid mean-value 
property holds for $\,\phi.\,$

     The general case will be proved 
by adapting the proof of G\aa rding 
\cite{lG50}. Let $\,{\gamma}(\xi, z),\;
\xi, z\in U\,$ with $\,\xi' \not = z',
\,$ be the Newtonian potential

\begin{equation}\label{E:Newt. potential}
{\gamma}(\xi, z) :\> = 
\begin{cases}
 {1 \over 2\pi} \log \|\xi' - z'\|^2, &
\text{if $ m = 1$,} \\
  \,{- 1 \over (2m -2)|\mathbb{S}|\|\xi'-z'\|^{2m-2}}, 
& \text{if $m > 1$.} 
\end{cases}
\end{equation}
Suppose that $\,\phi\in L^2_{{\rm loc}}
[X]\,$ and let $\,p : U = U_{[a]}
(r_0)\to U'= {\mathbb B}(r_0)\,$ 
be a pseudoball exhibiting the 
local forward square-integrability 
of $\,\phi.\,$ Let $\,G_1',\,G_2'\,$
be open sets with $\,G_1'\subset 
\subset G_2'\,$ where $\,\overline 
{G'_2}\subset U'_1 := {\mathbb B}
(r_1),\,r_1 < r_0.\,$ Choose $\,
\beta'\in C^\infty (U')\,$ such 
that $\,\beta = 0\,$ in a 
neighborhood of $\,\overline {G_1'}
\,$ and $\,\beta'\rfloor U\backslash 
\overline {G_2'} = 1.\,$ Let $\,G_j 
:= p^{-1}(G_j'),\,j = 1,2,\,$ and 
$\,\beta := p^*\beta'\,$ on $\,U.\,$ 
Then the function

\[
h(\xi) := {4\pi^m c_m\over (m-1)!}  
\>\int_{\overline G_2 \backslash G_1}
{\phi\,dd^c (\beta (z){\gamma}(\xi, z)) 
\wedge \upsilon_p^{m-1}},
\quad \forall\xi \in G_1,
\]
is well-defined. Choose a 
$C^\infty$-partition of unity 
$\,\{(U^j_\nu,\rho^j_\nu)\}\,$ 
on $\,V^j\cap \hat U\,$ in terms 
of the open sets $\,U^j_\nu\subseteq 
V^j.\,$ Let $\,{\gamma}'\,$ be 
the function induced by $\,{\gamma}
\,$ on $\,U'\times U'.\,$ For fixed 
$\,\xi\in G_1\,$ the form 
$\,\triangle_{z'} (\beta'(z')
{\gamma}'(\xi', z'))\,dv(z')\,$ is 
compactly supported in a neighborhood 
of $\,\overline G_2.\,$ The 
integrability of $\,\sum_{\nu}'
\hat\phi^{j}_{\nu}\,$ on $\,U'
\backslash \Delta'\,$ implies that 
of the function

\[
f_j(\xi',z') = \big({\sum_{\nu}}'
\hat\phi^{j}_{\nu}\big)(z')
\,\triangle_{z'}(\beta'(z')
{\gamma}'(\xi', z'))\hat\rho(\xi')
\]
is integrable on the product space 
$\,G'_1\times U'_{[a']}[r_1]$ for 
any $\,\hat\rho\in C^\infty (U'(r_1))$
(for $\,r_1 < r_0$). Thus the 
function $\,h,\,$ being equal to 
the sum of integrals (with respect to 
$\,z'$) over $\,U'\,$ of functions of 
the above type (with $\,\hat\rho 
= 1,j =1,\cdots, {\mathfrak s}_a$), 
belongs to $\,C^{{\mathfrak m}}
(G_1)\cap L^1_{{\rm loc}}(G^*_1),\,$ 
by the the Fubini's Theorem. 
Note that by virtue of the 
$L^2$-integrability of $\,
{\sum_{\nu}}'\hat\phi^{j}_{\nu}$
and the H\"older's inequality, 
one has, for $\,\xi\in G_1,\,$  

\[
\|h(\xi)\| \,
\le\,\sum_{j=1}^{{\mathfrak s}_a}[\int_{U'}
{\|{\sum_{\nu}}'\hat\phi^{j}_{\nu}
\big)(z')\|^2}\,dv(z')]^{1\over 2}\,
[\int_{U'}{\|\triangle_{z'}(\beta'(z')
{\gamma}'(\xi', z'))\|^2\,dv(z')}]^{1\over 2} 
\]
thus proving that $\,h\,$ is bounded 
on $\,G_1.\,$ Hence it follows from 
\cite[Lemma 4.2.12]{cT79} that $\,h\,$ 
is locally integrable in $\,G_1.\,$ 
Since $\,h\,$ is $C^\infty\,$ and 
harmonic in $\,G_1^*,\,$ it is 
semi-harmonic in $\,G_1.\,$ It will 
be shown that $\,\forall\eta \in 
C_0^\infty (G_1^*),\,$

\[
(h,\,\eta)_{G_1} = (\phi,\,\eta)_{G_1}.
\]
It suffices to prove the following:

\begin{equation}\label{E:action int.a}
\int_{G_1}{h\,\eta\,\upsilon_p^{m}} 
= \int_{G_1}{\phi\,\eta\,\upsilon_p^{m}}, 
\end{equation}
where $\,\eta\in C_0^\infty 
(U^{j_0}_{\nu_0}\cap G_1^*).\,$ 
By an interchange of the order of 
integration on the left-hand side 
of \eqref{E:action int.a}, this 
equation can be expressed 
equivalently as follows:

\begin{equation}\label{E:action int.b} 
\int_{G_1}{h(\xi)\,\eta(\xi)\,d\tilde\upsilon} 
 = \sum_{j=1}^{{\mathfrak s}_a}
\int_{\overline {G'_2}\backslash \Delta'}
{\big({\sum_{\nu}}'\hat\phi^{j}_{\nu}\big)(z')
\triangle_{z'}\psi(z')\,dv(z')},
\end{equation}
where 

\[
\psi(z') := \beta'(z')
\int_{(U^{j_0}_{\nu_0})'\cap G'_1}
{{\gamma}^{j_0}_{\nu_0}(\xi', z')\,
\hat\eta(\xi')\,dv(\xi')},
\quad z' \in \overline {G'_2},
\]
(where $\,{\gamma}^j_\nu\,$ is induced 
by $\,{\gamma}\,$ on $\,(U^j_\nu)'\times 
(U^j_\nu)',\,$ and the $\,\hat\eta\,$
by $\,\eta\,$ on $\,(U^{j_0}_{\nu_0})'$),
the operation being justified by 
the Fubini's Theorem (for the same 
reason as given above). Observe 
that

\[
\triangle_{z'} \psi (z') = - 
\triangle_{z'} (u(z')) + 
\triangle_{z'} (\int_{(U^{j_0}_{\nu_0})'\cap G'_1}
{{\gamma}^{j_0}_{\nu_0}(\xi', z')\,\hat\eta (\xi')
\,dv(\xi')}),
\]
where the function

\[
u (z') := (1-\beta'(z'))
\int_{(U^{j_0}_{\nu_0})'\cap G'_1}
{{\gamma}^{j_0}_{\nu_0}(\xi', z')\,
\hat\eta (\xi')\,dv(\xi')} 
\]
belongs to $\,C_0^\infty (U'_1).\,$ 
Thus

\begin{equation}\label{E:Lap. of psi}
\triangle_{z'} 
\psi(z') =
\begin{cases}
\hat\eta(z') - \triangle_{z'}(u (z')),
&\text{if $z'\in G'_1\cap (U^{j_0}_{\nu_0})'$,}\\
0 - \triangle_{z'} (u (z')), 
&\text{if $z' \in G'_1\backslash (U^{j_0}_{\nu_0})'$.} 
\end{cases}
\end{equation}
Hence the assertion \eqref{E:action 
int.a} follows from the relations 
\eqref{E:action int.b}-\eqref{E:Lap. 
of psi} and the weak harmonicity of 
$\,\phi.\,$ Consequently $\,h =\phi
\,$ almost everywhere in $\,G_1^*.
\,$ Since $\,G_1'\,$ is an arbitrary 
relatively compact open subset of 
$\,U',\,$ there exists a semi-harmonic 
function $\,\tilde\phi \in
C^{\beta\cap {\mathfrak m}}(X)
\cap C^\infty(X^*)\,$ such that 
$\,\tilde\phi =\phi\,$ a. e. 
in $\,X.\,$ Hence, by what has 
been proved, $\,\tilde\phi\,$ 
satisfies the local solid 
mean-value property in the 
domain of continuity of 
$\,\tilde\phi.\,$

   To prove the assertion "(2) 
$\Rightarrow$ (3)", let $\,U\,$ be a 
pseudo-ball at a point $\,a \in X^*\,$ 
of radius $\,r_0\,$ and $\,0< r < r_0.\,$ Without
loss of generality, assume that $\,\tilde \phi\,$ 
is real-valued. There exists a continuous function
$\,\hat \phi\,$ on the closed ball $\,
{\mathbb{B}_{[a']}[r]}\,$ such that
$\,p^*\hat \phi = \tilde\phi\,$ on $\,U\cap 
p^{-1}({\mathbb{B}_{[a']}[r]}).\,$ Also,
there exists a continuous function $\,h\,$ on $\,
{\mathbb{B}_{[a']}[r]}, \,$ harmonic in
$\,\mathbb{B}_{[a']}(r)$, such that $\,h(w) =
\hat \phi (w),\; \forall w \in \mathbb{S}_{[a']}(r).
\,$ Then the function $\,\psi : = \hat \phi - h\,$ has 
the solid mean-value property at $\,a'\,$ and vanishes 
on $\,\mathbb{S}_{[a']}(r).\,$ The same is true for 
the function $\,- \psi.\,$ Hence it follows from the 
maximum principle (Proposition \ref{P:max.prin.}) that 
$\,\psi \equiv 0 \,$ on $\,\mathbb{B}_{[a']}(r)$. 
Consequently $\,\phi\,$ is semi-harmonic in $\,X$.

\vskip 10 pt
    To prove the assertion "(3) $\Rightarrow$
(1)", let $\, U,\; U_0\,$ be pseudo-balls at 
$\,a\in X^*\,$ with $\,U_0 \subset \subset U 
\subset \subset X^*.\,$ It follows as in 
\eqref{E:weak Laplacian.a} that $\forall w 
\in C^0_0 (U_0'),\,$

\[
\int_{U_0'}{\hat \phi_{\varepsilon}
\,\triangle w\, dv} =
 {\rm const.}\int_{W}{dd^c \hat \phi 
\wedge (\alpha_\varepsilon^{(0)} * w)\, 
\upsilon_{[a']}^{m-1}},
\]
for sufficiently small $\,\varepsilon 
> 0,\,$ where the $\,\hat \phi\,$ 
being induced by $\,\phi\,$ on $\,U_0'.\,$ 
By the semi-harmonicity of $\,\phi,\,$ 
the above last integral vanishes. 
This shows that $\,\hat 
\phi_{\varepsilon}\,$ is harmonic in $\,U'\,$ 
for sufficiently small $\,\varepsilon > 0.\,$ 
Let $\,u : X \to [0,\infty)\,$ be a $C^2$-function with 
compact support in $\,X^*.\,$ By using a partition of
unity, it may be assumed that $\,{\rm Spt}(u)\,$ is 
contained in a pseudo-ball $\,U_0 \subset \subset X^*.
\,$ Then the following relation holds (for 
$\,\hat \phi_{\varepsilon}$ and the induced 
$\,\hat u \in C^2_c (U')\,$ of $\,u$):

\begin{equation}\label{E:weak Laplacian.b}
(\hat u,\,\triangle 
\hat \phi_{\varepsilon})_{U''}\,=\, 
(\hat \phi_{\varepsilon},\,
\triangle \hat u)_{U''} = 0
\end{equation}
It follows then from this relation 
(and the expression for the
semi-Laplace operator) that

\[
|\int_{U_0'}{\hat \phi\,dd^c
\hat u \wedge \upsilon_{[a']}^{m-1}}| \le 
\; {\rm Const.}\; \int_{U_0'}{|\hat \phi - 
\hat \phi_{\varepsilon}|\,
\upsilon_{[a']}^m}.
\]
The $L^1$-convergence of $\,\hat 
\phi_{\varepsilon}\,$ to $\,\hat 
\phi\,$ implies that

\[
\int_{U_0'}{\hat \phi \,dd^c \hat u\wedge 
\upsilon_{[a']}^{m-1}} = 0.
\]
Therefore $\,\phi\,$ is weakly harmonic 
in $\,X$. 
\end{proof}

\begin{remark}\label{R:MVP at a point} 
If $\,p : D\to \Omega\,$ is a Riemann 
domain, an element $\,\phi \in 
{\mathfrak H}_w(D)\cap L^2_{{\rm loc}}
[D]\,$ will be identified with its 
representative $\,\tilde\phi\in 
C^{\beta\cap {\mathfrak m}}(D)\cap 
C^\infty(D^*).\,$ If $\,\phi\,$ is 
locally integrable in $\,D\,$ and 
$\,\phi\in C^0 (D^*),\,$ then 
$\,\phi\,$ is semi-harmonic in 
$\,D\,$ iff in $\,D^*\,$ the local 
near harmonicity or the solid, 
resp. spherical, mean-value property 
holds for $\,\phi.\,$ 
\end{remark}

    The above Theorem gives an 
extension of the Weyl's Lemma 
(\cite{hW40}, pp.~415-416) to a 
Riemann domain:

\begin{corollary}\label{C:removable sing. shf.}
If $\,(X, p)\,$ is a Riemann domain, 
then

\[
{\mathfrak H}_w(X)\cap L^2_{{\rm loc}}
[X] = C^{\beta}(X)\cap C^\infty(X^*) 
\cap\ker\,(\triangle_p \rfloor X^*).
\]
\end{corollary}

   Theorem \ref{T:weak harm. and semi-harm.} 
and the maximum principle imply the following

\begin{corollary}\label{C:Dir.prob.}
Assume $\,(X, p)\,$ is a Riemann domain. 
Let $\,D\,$ be a domain in $\,X\,$ with 
$\,dD \not = \emptyset.\,$ Assume: {\rm i)} either 
$\,D \subseteq X^0\,$ or $\,D\,$ is irreducible; 
{\rm (ii)} $\eta_{_j} \in C^0 (\partial D)\,$  
for $\,j = 1,\, 2\,$ and $\,|\,\eta_{_1} - 
\eta_{_2}| < \varepsilon \,$ on $\,
\partial D;\,$ {\rm (iii)} $\phi = \phi_j\in C^0 
(\overline D), \,j = 1,\, 2,\,\,$ are bounded, 
weak solutions to the Dirichlet problem

\[
dd^c \,(\phi \,{\upsilon}_p^{m-1})
= \,\rho \;\; {\rm in} \; \, D \backslash A,
\quad \; \phi \,\rfloor \,d D = \eta_{_j},  
\]
where $\,A\,$ is thin analytic in $\,D,\,$ and
$\,\rho \in A^{2m,0}(D \backslash A)$. Then $\,
|\phi_1 - \phi_2| < \varepsilon \,$ on
$\,\overline D$.
\end{corollary}

\section{Euler and Neumann vector fields}\label{S:Euler v.f.}

   For each $\,a \in X,\,$ define $\, \rho  = 
\rho_{_a} := \|p^{[a]}\| : X \to \mathbb{R}.\,$  
The associated $\,\partial$-, respectively, $\bar 
\partial$-, {\it Euler vector field} (multiplied 
by $\,\rho_{_a}$) are given by

\begin{equation}\label{E:d-bar Euler vec.field}
E_{p,a} \> : = \,\rho_{_a} \> {\mathcal E}_{\rho_{_a}}; 
\quad \bar E_{p,a} \> := \,\rho_{_a} \> \bar 
{\mathcal E}_{\rho_{_a}} 
\end{equation}
in $\,X^*.\,$ It is easily shown that

\begin{equation}\label{E:d-bar Euler deri.}
\bar E_{p,a} \,=\, \sum_{j =1}^m 
(\bar p_j - \bar p_j(a)) \,
{\partial \over \partial \bar p_j} 
\quad {\rm in} \; X^*. 
\end{equation}
If $\,\phi\in C^1 (D)\,$ and $\,a \in D,\,$
the $\,a$-{\it radial derivative} of $\,\phi\,$ 
is given by

\begin{equation}\label{E:rad.deri.a}
({\mathfrak R}_{p,a} \,\phi) (z) :\>=
\,\partial_{_{\nabla \rho_{_a}}}
(\phi) \,(z),\quad z \in  D^*.
\end{equation}
Then $\forall z \in D^* \backslash  p^{-1}(a')$,

\begin{equation}\label{E:rad.deri.b}
({\mathfrak R}_{p,a} \,\phi) (z) = \, \sum_{j
=1}^m {(\tilde x_j - \tilde x_j(a))\, {\partial
\phi\over \partial \tilde x_j}\, +  \, (\tilde y_j
- \tilde y_j(a))\, {\partial \phi \over \partial
\tilde y_j} \over \| p^{[a]} (z)\|}
\big\rfloor_z 
\end{equation}
(cf. \cite{aD99}, p. 169).

     Let $\,{\rm j}_{_{dD}} : d D \to X,\,$ and
$\,{\rm j}_{a,r} : d D_a (r)\to X,\,$ denote
the inclusion mapping, for $\,a \in D^*\,$ and 
small positive $\,r.\,$  By tedious calculations it 
can be shown that

\begin{equation}\label{E:Euler deri.b}
{\rm j}_{a,r}^*(({i \over 2\pi})\,\bar \partial
\phi \wedge\upsilon_p^{m-1}) \, (z)
 = \, r^{2m-2}\, \bar E_{p,a} (\phi)
\rfloor_z \,\sigma_{a},\quad \forall z \in d D_a
(r). 
\end{equation}
It follows from the identities 
\eqref{E:rad.deri.a}-\eqref{E:Euler deri.b}
that, $\forall z \in D^*\backslash p^{-1}(a')$,

\begin{equation}\label{E:rad.deri.c}
{\rm j}_{a,r}^*(d^c \phi
\wedge\upsilon_p^{m-1}) \, (z)
 = \, {r^{2m-1}\over 2}\,({\mathfrak R}_{p,a}\,
\phi) (z)\,\sigma_{a}. 
\end{equation}
The following Proposition shows that the {\it normal 
derivative} of $\,\phi \in C^1 (\overline D)\,$ on 
$\,dD\,$ can be intrinsically defined: Let $\, d 
\sigma_{_{d D}}\,$ denote the (Lebesgue) surface measure 
on $\,dD\,$ induced by the local patches $\,p_{_U} : = 
p : U \to \mathbb{B}_{[a']}(r_0),\,$ at a point $\,a 
\in X^* \cap d D,\,$ and orientation-preserving 
diffeomorphisms of $\,\mathbb{B}_{[a']}(r_0).\,$

\begin{proposition}\label{P:normal deri.}
Let $\,\rho = 0\,$ be a local $\,C^1$ defining equation 
of $\,d D\,$ in a neighborhood $\,U \subseteq X^*$ of 
$\,a \in d D \cap X^*$ with $\,d \rho \not = 0\,$ on 
$\,d D \cap U.\,$ Then $\forall \phi \in C^1
(\overline D),\,$

\begin{equation}\label{E:normal deri.}
{\rm j}_{_{dD}}^*(d^c \phi \wedge
\upsilon_p^{m-1})\,= \, (-1)^{{m(m-1)\over
2}} {1\over 2 \|\mathbb{S}\|}
\partial_{\nu} \phi \,d\sigma_{_{d D}},
\end{equation}
where $\,\nu : = {\nabla \rho \over
\| \nabla \rho \|},\;\| \nabla \rho \|\,$ 
being the induced Euclidean norm of $\nabla 
\rho.\,$
\end{proposition}

\begin{proof} 
Let $\,d p_{[j]} : \> = d  p_1 \wedge \cdots \wedge 
d p_{j-1} \wedge d p_{j+1} \wedge
\cdots \wedge d p_m\,$ and $\,d \bar p : \>
= d \bar p_1 \wedge \cdots \wedge d \bar p_m.\,$ 
It can be shown (by tedious computations) that

\begin{equation}\label{E:loc.vol.form.a,b}
  \begin{split}
{\rm j}_{_{U \cap d D}}^*(dp_{[j]} \wedge d \bar
p) &  = \,2^m \,i^m\,(-1)^{m+j-1} \rho_j \, 
d\sigma_{_{U \cap d D}},\\
{\rm j}_{_{U \cap d D}}^*(dp \wedge d \bar
p_{[j]}) & = \,2^m \,i^m\,(-1)^{j-1} \rho_{\bar j} \, 
d\sigma_{_{U \cap d D}},
 \end{split}
\end{equation}
where $\,\rho_j : = (\partial \rho /
\partial p_j)\,\| \nabla \rho \|^{-1}\,$ and
$\,\rho_{\bar j} : = (\partial \rho /
\partial \bar p_j)\,\| \nabla \rho \|^{-1}.\,$ 
It follows from the definition \eqref{E:Euc. form}
that

\[
d^c \phi \wedge \upsilon_p^{m-1} = (-1)^{{(m-1)(m-2)
\over 2}} {({i\over 2})^m \over \|\mathbb{S}\|}
\sum_{j=1}^m \phi_{\bar p_j}(-1)^{m+j} dp_{[j]} 
\wedge d \bar p) - (-1)^{j-1}\phi_{p_j} dp \wedge 
d \bar p_{[j]}.
\]
From this the desired conclusion can be deduced by
making use of the identities 
\eqref{E:comp.dir.deri.}-\eqref{E:part.dir.deri.} 
and \eqref{E:loc.vol.form.a,b}. 
\end{proof}

     A relatively compact open set $\,G \subseteq X\,$ 
is called a {\it weak Stokes domain} in $\,X\,$ iff 
there exists a thin analytic set $\,A\,$ in $\,X\,$ 
containing the singular points of $\,X\,$ such that 
$\,\partial G \backslash A\,$ has locally finite 
${\mathfrak H}^{2m-1}$-measure and there exists an 
(oriented) boundary manifold $\,S\,$ of $\,G \cap 
{\mathcal R} (X)\,$ contained in $\,dG\,$ such that 
$\,\partial G \backslash (A \cup S)\,$ has zero 
${\mathfrak H}^{2m-1}$-measure. Since the set 
$\,d G \backslash (A \cup S)\,$ is a set of 
zero $\,{\mathfrak H}^{2m-1}$-measure, integration
over $\,S\,$ and $\,dG\,$ make no difference, 
provided one of them exists (and the notation $\,S\,$ 
will not be explicitly used).

     Owing to the normal derivative formula 
\eqref{E:normal deri.} and the equations 
\eqref{E:Dir. and dir. deri.}-\eqref{E:semi-Laplace 
operator}, the Stokes' theorem (\cite{cT79}, (7.1.3)) 
yields the following generalized {\it Green's first 
identity}:

\begin{lemma}\label{L:Green's identity}
If $\,G \subseteq X\,$ is a weak Stokes domain, then 
for all $\,\eta \in C^{\lambda}(\overline G)\,$ and 
$\,\phi \in C^{1,1}(\overline G),\,$

\begin{equation}\label{E:Gr.id.}
 [\eta,\bar \phi]_{_G} 
\, = \,\mathop{\int}_{d G}{\eta \, d^c \phi \wedge
\upsilon_p^{m-1}} \,-\, \mathop{\int}_G {\eta
\,dd^c \phi \wedge \upsilon_p^{m-1}}.
\end{equation}
\end{lemma}

\begin{proposition}\label{P:charac. semi-harm.a}
If $\,G \subseteq X\,$ is a weak Stokes domain, then, 
with respect to the Dirichlet product, 

   {\rm (1)}  ${\mathfrak H}(G) \cap C^{1,1}(G)\,= 
\, C^{1,1} (G)\cap (C^\lambda_0 (G;\mathbb{R}))^\perp;$ 

   {\rm (2)} ${\mathfrak H}(G) \cap C^{1,1}
(\overline G) \,= \, C^{1,1} (\overline G)\cap 
(C^{\lambda,(c)} (\overline G))^\perp.$ 
\end{proposition}

\begin{proof} 
  (1) Assume $\,\phi \in C^{1,1} (G)\cap 
(C^\lambda_0 (G;\mathbb{R}))^\perp.\,$
For each $\,a \in G^*\,$ choose a pseudo-ball 
$\,U \subset \subset G\,$ at $\,a\,$ of radius 
$\,r_0.\,$ The function $\,\eta_{a,r} : 
\overline G \to \mathbb{R},\; r \in (0,r_0),\,$ 
defined by

\[
\eta_{a,r} : =
\begin{cases} 
  \|p^{[a]}\|^2 - r^2, & \text{if $z \in 
   U_{[a]}(r)$,} \\  
   0, & \text{if $z \in \overline G \,\backslash
   \,\overline {U_{[a]}(r)}$},
\end{cases}
\]
is locally Lipschitz in $\,G.\,$ Thus $\,[\phi, 
\|p^{[a]}\|^2]_{a,r}= [\phi,\eta_{a,r}]_{_G} = 0.\,$ 
By the identity \eqref{E:rel.spher. and solid MV}, 
the condition that $\,[\phi, \|p^{[a]}\|^2]_{a,r} \, 
= \,0\,$ locally in $\,G^*\,$ (i.e. for small $\,r > 
0 $) is equivalent to $\,\phi\,$ being locally nearly 
harmonic in $\,G^*,\,$ hence, by Theorems
\ref{T:charac. semi-harm.} and \ref{T:weak harm. 
and semi-harm.}, also to $\,\phi\,$ being 
semi-harmonic in $\,G.\,$ The converse assertion 
follows from the Green's identity \eqref{E:Gr.id.}

  (2) Let $\, \phi \in {\mathfrak H}(G) \cap
C^{1,1} (\overline G).\,$ If $\,\xi \in C^\lambda 
(\overline G)\,$ and $\,\xi \rfloor \partial G = 
{\rm const.},\,$ the Green's identity 
\eqref{E:Gr.id.} implies that $\,[\bar \xi, \bar 
\phi]_{_G} = 0.\,$ Hence by the hermitian symmetry of 
the Dirichlet product (\cite{cT07}, (3.6)), 
$\, [\phi, \xi]_{_G} = 0.\,$ Conversely, if $\, w 
\in C^{1,1} (\overline G)\,$ such that $\, [w, \xi]_{_G} 
= 0,\;\forall \xi \in C^{\lambda}(\overline G)\,$ with 
$\,\xi \rfloor \partial G = 0,\,$ then by the assertion 
(1), $\,w\,$ is semi-harmonic in $\,G$.
\end{proof}

\begin{proposition}\label{P:Neum. prob.}
Assume that $\,G \subset X\,$ is a weak Stokes
domain with $\,dG \not = \emptyset,\; \rho 
\in A^{2m,0}(G \backslash A; \mathbb{R}),\,$ where 
$\,A\,$ is thin analytic in $\,G,\,$ and $\,\eta \in 
C^\lambda (\partial G; \mathbb{R}).\,$ If the Neumann 
problem

\begin{equation}\label{E:Neum. prob.}
dd^c \,(\phi \,{\upsilon}_p^{m-1})
= \rho \quad {\rm in} \;\, G \backslash A, \quad
\; \partial_\nu \phi = \eta \;\;{\rm on} \;\,d G
 \backslash A
\end{equation}
admits a real, weak {\rm (}resp. strong{\rm )} solution 
$\,\phi = \phi_0 \in C^1 (\overline G)\,$ 
{\rm (}resp. $\,C^{1,1}(\overline G)${\rm )}, 
then the set of all real, weak {\rm (}resp. strong{\rm )} 
solutions in $\,C^1 (\overline G)\,$ {\rm (}resp. 
$\,C^{1,1}(\overline G)${\rm )} of the equation
\eqref{E:Poisson equ.} on $\,G \backslash A\,$ subject
to the boundary condition

\begin{equation}\label{E:Neum. prob. bdry. cond.} 
\partial_{_\nu} \phi \ge \eta \quad ({\rm or}\;\;
\partial_{_\nu} \phi \le \eta)\quad {\rm on}\;\,
d G \backslash A
\end{equation}
is given by $\,\{\phi_0 + {\rm const.}\}$ 
\end{proposition}

\begin{proof} 
   Since every strong solution in $\,C^{1,1} 
(\overline G)\,$ to the equation \eqref{E:Poisson equ.} 
is a weak solution, it suffices to consider the case
where $\,\phi = \phi_1 \in C^1 (\overline G)\,$ is a real, 
weak solution to the equation \eqref{E:Poisson equ.}
satisfying the boundary condition 
\eqref{E:Neum. prob. bdry. cond.}. Set $\,\chi 
: \>= \phi_1 - \phi_0.\,$ Since $\,\chi\,$ is 
semi-harmonic in $\,G\,$ and, by the identity 
\eqref{E:normal deri.}, $\, d^c \chi \wedge 
{\upsilon}_p^{m-1}\,\rfloor \,d G \backslash A 
\ge 0\,$ (or $\,\le 0$), one has 

\[
\mathop{\int}_{d G}{d^c \chi \wedge \upsilon_p^{m-1}} 
= \mathop{\int}_G {dd^c \chi \wedge \upsilon_p^{m-1}} 
= \, 0.
\]
It follows that $\,d^c \chi \wedge {\upsilon}_p^{m-1}
\,\rfloor \,d G \backslash A \equiv 0.\,$ Then the 
Green's identity \eqref{E:Gr.id.} yields

\[
[\chi,\chi]_{_G} = \mathop{\int}_{d G}{\chi \, 
d^c \chi \wedge \upsilon_p^{m-1}} -
\mathop{\int}_G {\chi\,dd^c \chi \wedge
\upsilon_p^{m-1}} = \, 0.
\]
Hence $\,\|\bigtriangledown \chi \|^2 = 0\,$ a.e. 
in $\,G.\,$ Therefore the function $\,\chi\,$
is locally constant in $\,G^0.\,$ Consequently 
$\,\chi =$ constant on $\,\overline G.\,$ This 
completes the proof of the Proposition.
\end{proof}

\begin{remark}\label{E:complex Neum. prob.}
The above Proposition implies that, given $\,\rho 
\in A^{2m,0}(G \backslash A; \mathbb{C})\,$ and 
$\,\eta \in C^\lambda (\partial G; \mathbb{C}),\,$ 
an entirely similar assertion holds for the 
complex Neumann problem \eqref{E:Neum. prob.}.
\end{remark}

\begin{remark}\label{R:const.harm.func.}
Let $\,G \subseteq X\,$ be a weak Stokes domain 
with $\,dG \not = \emptyset.\,$ If $\,\phi \in C^1
(\overline G)\,$ is weakly harmonic in $\,G,\,$
then $\,\phi = {\rm const.}$ in $\,\overline G\,$ 
iff $\,\partial_\nu \phi \rfloor \, dG = 
{\rm const.}$
\end{remark}

\begin{example}\label{E:Newm. pr. for homog. poly.}
Let $\,P\,$ be a real homogeneous polynomial of degree 
$\,l\,$ in the variables $\,x_1, y_1\cdots, x_{2m}, 
y_{2m}.\,$ It is well-known that $\,P\,$ can be written

\[
P = \sum_{j=0}^l {\|x\|^{l-j} H_j(x)},
\] 
where $\,H_j\,$ is a harmonic homogeneous polynomial 
of degree $\,j,\,$ with $\,H_j \equiv 0\,$ whenever 
$\, l-j = {\rm odd}.\,$ The $\,H_j\,$ can be calculated 
by an effective algorithm {\rm (see} \cite{AR95}{\rm )}. 
Let $\,\pi : X \to \mathbb{C}^m\,$ be an analytic covering 
with sheet number $\,s.\,$ Set $\,\tilde P =  \pi^* P
\,$ and $\,G : = \{z \in X\,|\,\|\pi(z)\| < 1\}.\,$ 
It follows from the identity \eqref{E:normal deri.} and 
{\rm Proposition \ref{P:Neum. prob.}} that the set of all the 
strong solutions in $\,C^{1,1}(\overline G)\,$ of the
Neumann problem{\rm :}

\[
dd^c \,(\phi \,{\upsilon}_\pi^{m-1}) 
= 0 \quad {\rm in} \;\, G^*, \quad
\; \partial_{_\nu} \phi \ge \tilde P - H_0 \;\;
{\rm on} \;\,d G \cap G^*,
\]
is given by $\,\{\psi + {\rm const.}\},\,$ where 
$\,\psi : = \sum_{j=1}^l {\pi^*H_j/j}.\,$ 
Consequently,

\[
\mathop{\int}_{d G}
{\tilde P \, d \sigma_{_{d G}}} = 
\begin{cases}
s \,H_0\, \,|\mathbb{B}|, & \;
\text{if $\,l = {\rm even}$,} \\
0, & \; \text{if $\,l = {\rm odd}$.} 
\end{cases}
\]
\end{example}

      In \cite {hW13}, p. 182, Weyl gave an 
alternative definition of the Laplace operator in 
terms of the Gauss' divergence theorem. In this 
light it makes sense to define, in view of the 
identity \eqref{E:normal deri.}, the {\it harmonic 
residue} of a function $\,\phi \in C^1 (D \backslash 
\{a\})\,$ at a point $\,a \in D^0\,$ as follows:

\begin{equation}\label{E:harm.res.}
   Res_a (\phi,r) :\> = 
  \begin{cases} 
  \mathop{\int}_{d D_{[a]}(r)}{(- d^c) \phi},
   & \text{$m = 1$,} \\
  {1 \over m-1} 
  \mathop{\int}_{d D_{[a]}(r)}{(-d^c)
  \phi\wedge\upsilon_p^{m-1}}, & \text{$m >
1$,} 
  \end{cases}
\end{equation}
for small $\,r > 0$ (cf. B\^ocher \cite{mB95} (see also 
\cite{gE28}), and \cite{ABR01}, pp. 213-214). If $\,\phi
\,$ is a semi-harmonic function with an isolated 
singularity at $\,a,\,$ then the definition 
\eqref{E:harm.res.} is independent of the pseudo-radius 
(as the Stokes' theorem easily shows).

\begin{example}\label{E:harm.res.} 
Let $\,p : X \to \Omega\,$ be a Riemann domain 
and $\,h \in {\mathfrak H}(X).\,$ Let $\,
a \in X\,$ and $\,\alpha, \,s \in [0, \infty)\,$
be constants. Define $\,\phi : X \backslash
p^{-1}(p(a)) \to \mathbb{C}\,$ by

\[
 \phi (z) :\> = \,{(\log \|p^{[a]}\|^2)^\alpha
 \,h(z) \over \|{p^{[a]}(z)}\|^{2m-2 +s}}. 
\]
Let $\,U \,$ be a pseudo-ball at $\,a\,$ of 
radius $\,r_0 > 0$. Then $\forall r \in
(0,r_0)$,

\begin{equation}\label{E:1-dim.res.}
Res_a (\phi,r) = \,\big[ - {\alpha \,(\log
\,r^2)^{\alpha -1} \over r^s} +
{{s\over 2} \over r^s} \,(\log
\,r^2)^\alpha \big]\,\nu_p(a)
\,h(a), \quad m = 1,
\end{equation}

\begin{equation}\label{E:m-dim.res.}
\hskip 4 pt Res_a (\phi,r) = \,\big[ - {\alpha
\,(\log \,r^2)^{\alpha -1} \over  (m-1)r^s} +
{m-1+{s\over 2} \over (m-1)\,r^s} \,(\log
\,r^2)^\alpha \big]\,\nu_p(a)
\,h(a), \quad m > 1. 
\end{equation}
\end{example}

\begin{proof}
  Let $\,k = {s \over 2}$, and $\, r \in
(0,r_0)$.  Then $\forall \alpha > 0\,$ one has

\[
  \begin{split} 
    {{\rm j}_{a,r}}^* (d^c \phi)\, & = \,{(\log\,r^2)^\alpha\, 
{{\rm j}_{a,r}}^* d^c h \over r^{2m-2 + 2k}} \,
  + \,{\alpha\,(\log\,r^2)^{\alpha -1}{{\rm j}_{a,r}}^* 
(h\,d^c \|p^{[a]}\|^2) \over r^{2m + 2k}}\\
  & \, - (m-1 + k)\,{(\log \,r^2)^\alpha\, {{\rm j}_{a,r}}^* 
(h \,d^c (\|{p^{[a]}(z)}\|^2)) \over r^{2m + 2k}}.
  \end{split}
\]
It follows from this relation and the
semi-harmonicity of $\,h\,$ that, upon integrating 
the form $\,d^c \phi\wedge \upsilon_p^{m-1}\,$ 
over $\,dU_{[a]}(r),\,$ the first term vanishes,
the second and the third term yield the number 
$\,{\alpha \,(\log \,r^2)^{\alpha -1} \over r^s}\,
\nu_p(a) \,h(a),\,$ respectively, $\,{-(m-1+ k)\,
(\log \,r^2)^\alpha \over r^s}\,\nu_p(a)\,h(a),\,$ 
by invoking the identity \eqref{E:rel.spher. and solid MV} 
and the mean-value properties of $\,h.\,$ The case 
$\,\alpha = 0\,$ is similar. Hence the relations 
\eqref{E:1-dim.res.}-\eqref{E:m-dim.res.} are proved. 
\end{proof}

\begin{proposition}\label{P:res. and semi-harm.}
A locally integrable function $\,\phi\,$ in $\,D\,$ 
is semi-harmonic iff $\,\phi \in C^1 (D^*)$ and there
exists at each $\,a \in D^*\,$ a pseudo-ball $\,U
\subseteq D^*\,$ {\rm (}of radius $\,r_a${\rm )}
such that 

\begin{equation}\label{E:vani. res.}
\,Res_a (\phi,r) = \,0, \quad \forall r \in 
(0, r_a).
\end{equation}
\end{proposition}

\begin{proof}   
   If $\,\phi \,$ is semi-harmonic in $\,D$, then
the relation \eqref{E:vani. res.} holds at each point 
$\,a\in D^*\,$ by Corollary \ref{C:removable sing. shf.} 
and the Stokes theorem. To prove the 
converse, assume $\,U\,$ is a pseudo-ball at 
$\,a \in D^*\,$ of radius $\,r_a\,$ satisfying the 
condition \eqref{E:vani. res.}. Let $\,V^j\,$ and
$\,\hat \phi_j,\, 1 \le j \le s,\,$ be the same
as in the proof of Lemma \ref{L:const. aver.}. 
Denoting by $\,{\mathfrak j}_r,\; r \in
(0,r_a),\,$ the map: $\mathbb{S} = \mathbb{S} (1)
\to \, \mathbb{S}_{[a']}(r), \;{\mathfrak j}_r
({\mathfrak z}) = a' + r {\mathfrak z},\,$ one has

\[
[\phi \rfloor V^j]_{a,\rho} =
\mathop{\int}_{\mathbb{S}}{\hat \phi_j (a'
+ \rho \mathfrak z) \,{\mathfrak
j}_\rho^*(\sigma_{a'})} =
\mathop{\int}_{\mathbb{S}}{\hat \phi_j (a' +
\rho \mathfrak z) \,\sigma_0}.
\]
On the other hand,

\[
{d \over d\rho}\, \big([\phi \rfloor
V^j]_{a,\rho} \big)
\big\rfloor_{\rho = r}\,= \, \mathop{\int}_{\mathbb{S}}{
(\partial_{\mathfrak z}\,\hat \phi_j)\,(a' + r
{\mathfrak z}) \,\sigma_0}. 
\]
Therefore the relations \eqref{E:rad.deri.b}, 
\eqref{E:rad.deri.c} and \eqref{E:harm.res.}
imply that

\[
  {d \over d\rho} \, [\phi \rfloor
  U]_{a,\rho}\,\big\rfloor_{\rho = r} =
  \begin{cases} 
  ({- 4\pi \over r})\,Res_a (\phi,r),
  & \text{$m =1$,}  \\ 
  {2(1-m)\over r^{2m-1}}\, Res_a (\phi,r),
  & \text{$m > 1$.}
  \end{cases}
\]
Hence $\,[\phi \rfloor U]_{a,r} = {\rm
const.}, \; \forall r \in(0, r_a).\,$
Now the semi-harmonicity of $\,\phi\,$ follows from 
Lemma \ref{L:near harm.}, Theorem \ref{T:charac. semi-harm.}
and Remark \ref{R:MVP at a point} to Theorem 
\ref{T:weak harm. and semi-harm.}.
\end{proof}

   Another important case of an Euler type vector 
field associated with a smooth boundary manifold is 
the $\bar \partial$-Neumann vector field. For the 
definition, let $\,\rho = 0\,$ be a local $C^1$ 
defining equation of $\,d D\,$ in an open set $\,U 
\subseteq X^*\,$ with $\,d \rho \not = 0\,$ on $\,d D 
\cap U.\,$ Define in $\,U\,$ the 
$\bar \partial$-{\it Neumann vector field}

\begin{equation}\label{E:Neum.vec.}
\bar \partial_n \>:= \, {1 \over
\| \nabla \rho \|} \, \bar {\mathcal E}_{\rho}. 
\end{equation}
It follows from the relations \eqref{E:Euler vec.field} 
and \eqref{E:Neum.vec.} that, $\forall\phi \in C^1 
(\overline U)$,

\begin{equation}\label{E:Neum.deri.}
(\bar \partial_n \phi) (\zeta) \, = \, 2 
\sum_{j=1}^m {\rho_j (\zeta )\,  {\partial \phi
\over\partial \bar p_j} (\zeta)},\quad \zeta \in
U \cap dD.
\end{equation}
(For an alternative definition of the $\bar 
\partial$-Neumann derivative, see \cite{aK95}, 
p. 62). The formula \eqref{E:Neum.deri.} yields 
a derivative of $\,\phi\,$ along a (complex) 
direction in the complex line passing through 
the unit outward normal to $\,U \cap d D\,$ at 
$\,\zeta.\,$ It is intrinsically defined. This 
can be seen as follows. Consider the $\,
(m-1,m)$-form

\begin{equation}\label{E:mu-form}
\mu_{_\phi} : = \,\sum_{k=1}^m {(-1)^{m+k-1} 
({\partial \phi \over \partial \bar p_k})\,
d p_{_{[k]}} \wedge d \bar p}
\end{equation}
(\cite{aK95}, p. 2). Using the first identity in 
\eqref{E:loc.vol.form.a,b}, it is easy to deduce 
from the definition \eqref{E:mu-form} 
and the equation \eqref{E:Neum.deri.} the 
following:

\begin{lemma}\label{L:Intri.Neum.deri.}
\begin{equation}\label{E:Intri.Neum.deri.}
 (\bar \partial_n \phi) \, d\sigma_{_{U \cap d
D}} \,= \,2^{1-m}i^{-m} \,{\rm j}_{_{U \cap d
D}}^*\,\mu_{_\phi}.
\end{equation}
\end{lemma}

     This relation implies that the definition 
\eqref{E:Neum.vec.} is independent of the choice of 
the local defining equation of $\,d D.\,$ Some
applications of the $\bar \partial$-Neumann derivative 
are given in \cite{aK95} and \cite{cT07}.

\end{document}